\documentclass{amsart}

\usepackage[letterpaper,margin=1in]{geometry}
\usepackage{mathtools, verbatim, rotating, comment, amsmath, amsfonts, amsthm, amssymb, graphicx, tablefootnote}
\usepackage[dvipsnames]{xcolor}
\usepackage[shortlabels]{enumitem}
\usepackage{url,hyperref,multirow}
\usepackage[matrix,arrow]{xy}

\numberwithin{equation}{section}

\newtheorem{thm}{Theorem}[section]
\newtheorem{prop}[thm]{Proposition}
\newtheorem{conj}[thm]{Conjecture}
\theoremstyle{definition}
\newtheorem{defn}[thm]{Definition}
\theoremstyle{remark}
\newtheorem{rem}[thm]{Remark}
\DeclareMathOperator{\coker}{coker}
\newcommand{\prim}{\mathrm{prim}}
\DeclareMathOperator{\Pic}{Pic}
\DeclareMathOperator{\Sym}{Sym}
\DeclareMathOperator{\sym}{sym}
\DeclareMathOperator{\MS}{MS}
\newcommand{\GL}{\mathrm{GL}}
\newcommand{\SL}{\mathrm{SL}}
\newcommand{\Prim}{\mathrm{Prim}}

\newcommand{\OUT}{\mathrm{out}}
\newcommand{\IN}{\mathrm{in}}
\newcommand{\END}{\mathrm{END}}

\newcommand{\PP}{\mathbb{P}}
\renewcommand{\SS}{\mathbb{S}}
\newcommand{\ZZ}{\mathbb{Z}}

\newcommand{\C}{\mathcal{C}}
\renewcommand{\L}{\mathcal{L}}
\newcommand{\N}{\mathcal{N}}
\newcommand{\OO}{\mathcal{O}}
\newcommand{\R}{\mathcal{R}}

\newcommand{\cross}{\times}
\newcommand{\tensor}{\otimes}

\newcommand{\longto}{\mathop{\longrightarrow}\limits}
\newcommand{\textand}{\quad \text{and} \quad}

\renewcommand{\to}{\mathop{\rightarrow}\limits}
\newcommand{\size}[1]{\lvert #1 \rvert}

\newcommand{\floor}[1]{\left\lfloor #1 \right\rfloor}
\newcommand{\Bigfloor}[1]{\Big\lfloor #1 \Big\rfloor}
\newcommand{\ceil}[1]{\left\lceil #1 \right\rceil}
\newcommand{\Bigceil}[1]{\Big\lceil #1 \Big\rceil}

\newcommand{\isom}{\cong}

\renewcommand{\(}{\left(}
\renewcommand{\)}{\right)}

\newcommand{\ignore}[1]{}
\newcommand{\ds}{\displaystyle}

\renewcommand{\epsilon}{\varepsilon}
\newcommand{\quickcol}[1]{%
  \begin{tabular}{@{}c@{}}%
    #1
  \end{tabular}
}

\begin{document}

\title[%
  Cohomology of the incidence correspondence
]{%
  Even-carry polynomials\\%
  and cohomology of the incidence correspondence\\%
  in positive characteristic
}
\author{%
  Evan M. O'Dorney
}
\begin{abstract}
We consider the cohomology groups of line bundles $\mathcal{L}$ on the \emph{incidence correspondence}, that is, a general hypersurface $X \subset \mathbb{P}^{n-1} \times \mathbb{P}^{n-1}$ of degrees $(1,1)$. Whereas the characteristic $0$ situation is completely understood, the cohomology in characteristic $p$ depends in a mysterious way on the base-$p$ digits of the degrees $(d, e)$ of $\mathcal{L}$. Gao and Raicu (following Linyuan Liu) prove a recursive description of the cohomology for $n = 3$, which relates to Nim polynomials when $p = 2$. In this paper, we devise a suitable generalization of Nim polynomials, which we call \emph{even-carry polynomials,} by which we can solve the recurrence of Liu--Gao--Raicu to yield an explicit formula for the cohomology for $n = 3$ and general $p$. We also make some conjectures on the general form of the cohomology for general $n$ and $p$, for which a recurrence relation was recently derived by Kyomuhangi--Marangone--Raicu--Reed.

\end{abstract}

\maketitle


\emph{Acknowledgement:} The Version of Record of this manuscript has been published and is available in Experimental Mathematics, April 1, 2025, \url{http://www.tandfonline.com/10.1080/10586458.2025.2481257}


\section{Introduction} \label{sec:intro}

A broad and longstanding area of research in algebraic geometry is to compute the cohomology $H^i(X, \L)$ of line bundles on varieties $X$ over a field $\Bbbk$. This is difficult in general, especially when $\Bbbk$ has positive characteristic $p$, as there are fewer general theorems to apply.

Following Gao--Raicu \cite{GaoRaicu}, we look at the \emph{incidence correspondence}, that is, the hypersurface $X \subset \PP^{n-1} \cross \PP^{n-1}$ given by the vanishing of a single nondegenerate $(1,1)$-form
\[
  \omega(\vec z, \vec w) = z_1 w_1 + \cdots + z_n w_n.
\]
This is a natural variety to study; so natural, indeed, that its cohomology groups have been studied multiple times in different guises, as we will mention below. In characteristic zero, the cohomology is governed by the Borel--Weil--Bott theorem, which fails in positive characteristic.

We have $\Pic X \isom \ZZ^2$, with the inclusion $\iota : X \to \PP^{n-1} \cross \PP^{n-1}$ inducing an isomorphism on Picard groups. By symmetry considerations and known results (Serre duality, Kempf vanishing), Gao--Raicu show \cite{GaoRaicu} that the only degrees in which cohomology can occur are $0$, $n-2$, $n-1$, and $2n-1$ and that all cohomology of line bundles on $X$ is determined by the values of
\[
  K(d, e) \coloneqq H^{n-1}(X, \L_{d,e}) \quad \text{where} \quad  \L_{d,e} = \iota^*\(\OO_{\PP^{n-1} \cross \PP^{n-1}}(e, 1 - d - n)\), \quad d, e \geq 0.
\]
The group $\SL_n$ acts on $X$, preserving each element of the Picard group, and thus $K(d, e)$ is a finite-dimensional representation of $\SL_n$. Hence it is natural to consider its character $\kappa(d, e)$, which lies in the character ring
\[
  A = \big(\mathbb{Z}[x_1,\ldots,x_n]/(x_1 \cdots x_n - 1)\big)^{S_n}.
\]
The character $\kappa(d,e)$ determines the representation $K(d,e)$ up to semisimplification.

Gao--Raicu provide a recurrence relation for computing the character $\kappa(d,e)$ in the case $n = 3$, originally due to Linyuan Liu who studies these representations in another setting \cite[Th\'eor\`eme 1 and Proposition 9]{Liu2023}. Explicit, non-recursive formulas for $\kappa(d,e)$ have only been found when $p = 2$, due to Gao \cite{GaoThesis} in the cases $n = 3, 4$ and to Kyomuhangi--Marangone--Raicu--Reed \cite{KMRR24} for general $n$. In these cases, the formulas are in terms of \emph{Nim polynomials,} which are sums over integer triples satisfying a condition on their binary digits arising in the game of Nim.

The main result of this paper (Theorem \ref{thm:Even_Carry}) is an explicit, non-recursive formula for $\kappa(d,e)$ for $n = 3$ and all $p$. This provides an explicit solution to the recurrence of Liu--Gao--Raicu. The key advance is an appropriate generalization of Nim polynomials for general $p$. We call them \emph{even-carry polynomials,} because their terms are indexed by integer triples whose addition in base $p$ involves only carrying even digits ($0$ and $2$).

We naturally wonder what happens for larger $n$. Unfortunately, the known recurrence \cite[Theorem 1.1]{KMRR24} does not make obvious what should be the suitable generalization of Nim/even-carry polynomials. To this end, we carve out suitable pieces of $\kappa(d,e)$, which we call \emph{primitive cohomology polynomials,} or \emph{Prim polynomials} $\Prim(m)$ for short (Section \ref{sec:conj}). We list some known examples and conjectural patterns in the Prim polynomials, as steps toward computing $\kappa(d,e)$ for general $n$ and $p$.

\section{Representations of \texorpdfstring{$\SL_n$}{SL\textunderscore n} in characteristic \texorpdfstring{$p$}{p}}

\subsection{Schur functions} \label{sec:schur}
If $W$ is a vector space, we denote by $S^m W$ and $D^m W$ the $m$th symmetric and divided powers of $W$, respectively, which are related by
\[
D^m(W) = \(S^m(W^*)\)^*.
\]

If $\lambda = (\lambda_1,\ldots,\lambda_d)$ is a partition, we denote by $s_\lambda \in A$ the character of the representation of $\SL_n$ obtained by applying the corresponding Schur functor $\SS_\lambda$ to the standard representation $V = \Bbbk^n$. We have $s_\lambda \neq 0$ exactly when $\lambda$ has at most $n$ (nonzero) parts. Trailing zero parts in $\lambda$ have no effect. For example:
\begin{itemize}
  \item $s_{()} = 1$ (the empty partition)
  \item $s_1 = [V] = x_1 + \cdots + x_n$
  \item $s_d = [\Sym^d V] = [D^d V] = \ds \sum_{a_1 + \cdots + a_n = d} \!\!\vec {\ t\ \!}^{\!\vec a}$ % Default spacing is bad
  \item $\ds s_{\underbrace{\scriptstyle1,\ldots,1}_{d}} = [\Lambda^d V] = \sum_{i_1 < \cdots < i_d} x_{i_1} \cdots x_{i_d}$
  \item $s_{\underbrace{\scriptstyle1,\ldots,1}_{n}} = [\Lambda^n V] = x_1\cdots x_n = 1$. More generally, if $\lambda$ is written with $n$ entries, then adding $1$ to each entry of $\lambda$ has no effect.
\end{itemize}
In view of the last example, we can extend the definition of $s_\lambda$ to arbitrary weakly decreasing sequences $\lambda_1 \geq \cdots \geq \lambda_n$ of integers, not necessarily positive, by shifting all entries upward. The $s_\lambda$ with $\lambda_n = 0$ form a basis of the character ring $A$. When $\lambda$ is not weakly decreasing, we set $s_\lambda = 0$.

In characteristic $0$, the $\SS_\lambda V$ are precisely the irreducible representations of $\SL_n$. In positive characteristic, things are more subtle: each $\SS_\lambda V$ is an indecomposable, but not necessarily irreducible, representation whose socle is the irreducible $\SL_n$-module of highest weight $\lambda$.

We use a superscript $\vee$ to denote the involution on $A$ coming from sending each generator $x_i$ to its inverse $x_i^{-1} = x_1\cdots \widehat{x_i} \cdots x_n$. This corresponds to dualizing the corresponding representation.

If $W$ is a representation of $\SL_n\Bbbk$, we let $F(W)$ be the Frobenius twist of $W$ (where elements act by their entrywise $p$th powers). By extension, we use $F$ to denote the Frobenius map on polynomial rings (such as $S, R, A$) that sends each generator $z_i, w_i, x_i$ to its $p$th power. We write $F^k$ for the $k$-fold composition of $F$ (denoted by $F^{p^k}$ in \cite{GaoRaicu}).

\subsection{Truncated Schur functions} \label{sec:trunc}
Let $q = p^r$. Another useful element of $A$ is the \emph{truncated} symmetric function
\[
s_d^{(q)} = \ds \sum_{\substack{a_1 + \cdots + a_n = d \\ 0 \leq a_i < q}} {\vec x}^{\vec a},
\]
the character of the cokernel
\[
S_d^{(q)}V = \coker(F^r(V) \tensor S^{d - q} V \to S^d V).
\]
As a basis of $F^r(V)$ is a regular sequence in $S^\bullet V$, the Koszul complex offers a resolution
\[
  \cdots \to S_d^{(q)}V \to \Lambda^3 F^r(V) \tensor S^{d - 3q} V \to \Lambda^2 F^r(V) \tensor S^{d - 2q} V \to F^r(V) \tensor S^{d - q} V \to S^d V \to S_d^{(q)}(V),
\]
yielding the formula
\begin{equation} \label{eq:formula for trunc_sym}
  s_d^{(q)} = \sum_{i = 0}^{n} (-1)^i F^r \big(s_{\underbrace{\scriptstyle1,\ldots,1}_{i}}\big) \cdot s_{d - iq}.
\end{equation}

Taking their cue from the Jacobi--Trudi identity
\[
s_{a,b} = s_a s_b - s_{a+1} s_{b-1},
\]
Gao and Raicu define, following Walker \cite[\textsection 1]{Walker94}, the \emph{truncated Schur functions}
\begin{equation} \label{eq:def_trunc_s}
  s_{a,b}^{(q)} = s_a^{(q)} s_b^{(q)} - s_{a+1}^{(q)} s_{b-1}^{(q)}.
\end{equation}
Note the dualities
\begin{equation} \label{eq:trunc_dual}
  s_a^{(q)\vee} = s_{n(q-1) - a}^{(q)}, \quad s_{a,b}^{(q)\vee} = s_{n(q-1) - b, n(q-1) - a}^{(q)}.
\end{equation}
However, the characters $s_{a,b}^{(q)}$ are not so well-behaved as $s_a^{(q)}$: the ``truncation'' may actually \emph{increase} the highest weight. For instance, let $n = 3$ and display elements of $A$ in a shorthand triangle form so that, for instance, $\quickcol{4 \\ 5\ \ 7}$ means $4x_1 + 5x_2 + 7x_3$. For $p = 2$, $n = 3$, we have
\[
s_{3,2} = \quickcol{
  1\ \ 1\ \ 1 \\
  1\ \ 2\ \ 2\ \ 1 \\
  1\ \ 2\ \ 1 \\
  1\ \ 1%
} \qquad \text{but} \qquad 
s_{3,2}^{(4)} = \quickcol{
  1 \\
  1\ \ 1 \\
  1\ \ 1\ \ 1 \\
  1\ \ 2\ \ 2\ \ 1 \\
  1\ \ 1\ \ 2\ \ 1\ \ 1 \\
  1\ \ 1\ \ 1\ \ 1\ \ 1\ \ 1\rlap{,}%
}
\]
and other truncations have striking nonconvex shapes, such as (letting $p = 5$, $n = 3$)
\[
s_{6,6}^{(5)} = \quickcol{
  1 \\
  1\ \ 1 \\
  1\ \ 1\ \ 1\ \ 1\ \ 1\ \ 1\ \ 1 \\
  1\ \ 1\ \ 1\ \ 1\ \ 1\ \ 1 \\
  1\ \ 1\ \ 1\ \ 1\ \ 1 \\
  1\ \ 1\ \ 1\ \ 1\ \ 1\ \ 1 \\
  1\ \ 1\ \ 1\ \ 1\ \ 1\ \ 1\ \ 1\rlap{.} \\
  1\ \ 1 \\
  1%
}
\]
Such shapes do \emph{not} appear in the observed data for $\kappa(d,e)$, and so we must use $s_{a,b}^{(q)}$ cautiously, only in the ranges of values for $a$ and $b$ for which it is well behaved. Such is the case when both $a$ and $b$ are less than $q$ (because the truncation has no effect) or both are at least $(n-1)(q-1)$ (by \eqref{eq:trunc_dual}) or when $a$ is large and $b$ is small: % (see Proposition \ref{prop:trunc} below):
\begin{prop} \label{prop:trunc}
  If $a = n(q-1) - a'$ with $a', b \leq q-1$, then
  \[
    s_{a,b}^{(q)} = s_{a' + b, \underbrace{\scriptstyle a', \ldots, a'}_{n-2}}.
  \]
\end{prop}
\begin{proof}
  By \eqref{eq:trunc_dual}, we have
  \begin{align*}
    s_{a,b}^{(q)} &= s_a^{(q)} s_b^{(q)} - s_{a+1}^{(q)} s_{b-1}^{(q)} \\
    &= s_{a'}^\vee s_b - s_{a' - 1}^\vee s_{b-1}.
  \end{align*}
  The desired identity then reduces to 
  \begin{equation} \label{eq:under}
    s_{b, \underbrace{\scriptstyle 0, \ldots, 0}_{n-2}, -a'} = s_{a'}^\vee s_b - s_{a' - 1}^\vee s_{b-1}.
  \end{equation}
  This identity is a rather well-known special case (see Grinberg \cite[Corollary 4.31]{GrinbergPelletier}) is a consequence of the tensor product formula due to Koike \cite[equation (0.3)]{KoikeDecomp}:
  \begin{equation} \label{eq:Koike}
    s_{\lambda_1,\ldots,\lambda_r, \underbrace{\scriptstyle 0, \ldots, 0}_{n-r-s}, -\mu_s, \ldots, -\mu_1} = \sum_{\nu} (-1)^{\size{\nu}} s_{\lambda/\nu} s_{\mu/\nu^\top}^\vee,
  \end{equation}
  where $\nu$ runs over partitions, $\nu^\top$ is the conjugate of $\nu$, and $s_{\lambda/\nu}$ is the skew Schur polynomial. In our instance, $\lambda$ and $\mu$ have only one part, so the only two terms are $\nu = ()$ and $\nu = (1)$, verifying \eqref{eq:under}.
\end{proof}

\subsection{Interpretations of \texorpdfstring{$K(d,e)$}{K(d,e)}}
In this subsection, we recall the existing results from \cite{GaoRaicu}, which include an effective algorithm to compute $K(d,e)$. Let $\Bbbk$ be a field and $V = \Bbbk^n$, $n \geq 3$. 
We work on $\PP^{n-1} = \PP(V)$ and $\PP^{n-1} \cross \PP^{n-1} = \PP(V) \cross \PP(V)^\vee$. We use the tautological exact sequence
\begin{equation} \label{eq:R_res}
  0 \to \R \to V \tensor \OO_{\PP^{n-1}} \to \OO_{\PP^{n-1}}(1) \to 0
\end{equation}
where $\R$ is the tautological subbundle of rank $n-1$. Then for $j \geq n - 2$ (we are only concerned with $j = n-2$), we have \cite[Theorem 1.3]{GaoRaicu}
\[
  H^j(X, \L_{d,e}) \isom H^{j - n + 2}(\PP^{n-1}, D^d \R(e - 1)).
\]
Applying the divided power functor $D^d$ to the resolution \eqref{eq:R_res} gives a resolution of the divided powers $D^d\R$, from which we get an exact sequence:
\[
\xymatrix@C=1.5em{
 0 \ar[r] & H^0(\PP^{n-1}, D^d \R(e-1)) \ar@{=}[d] \ar[r] &
 D^d V \tensor S^{e-1} V \ar[r] & D^{d-1} V \tensor S^{e} V \ar[r] & H^1(\PP^{n-1}, D^d \R(e-1)) \ar@{=}[d] \ar[r] & 0
  \\
 & K(d,e) & & & K(e,d)^\vee
}
\]
In other words, studying $K(d,e)$ is tantamount to understanding the kernel (or equivalently, the cokernel) of the natural map
\[
\mu_{d,e} : D^d V \tensor S^{e-1} V \to D^{d-1} V \tensor V \tensor S^{e-1} V \to D^{d-1} V \tensor S^{e} V.
\]
Note that $\mu_{d,e}$ is equivariant with respect to $\SL(V)$, indeed $\GL(V)$, and hence induces the representation-theoretic structure of the desired sheaf cohomology over either one of these groups. Its character $\kappa(d,e) \in A$ may be represented by a symmetric, homogeneous polynomial of degree $d + e - 1$ in $\mathbb{Z}[x_1,\ldots,x_n]$.

Unraveling the definitions, we come up with the following completely explicit description of $K(d,e)$:%, which we have coded in Macaulay2:

\begin{prop}
Let $d$ and $e$ be positive integers. Consider the matrix $M$ over $S = \Bbbk[z_1, z_2, \ldots, z_n]$ such that
\begin{itemize}
  \item the rows are indexed by the monomials $f_1, \ldots, f_r$ of degree $d - 1$ in $S$
  \item the columns are indexed by the monomials $g_1, \ldots, g_s$ of degree $d$ in $S$
  \item the entries are
  \[
    M_{ij} = \begin{cases*}
      \dfrac{g_j}{f_i} & if $f_i \mid g_j$ \\
      0 & otherwise.
    \end{cases*}
  \]
\end{itemize}
Then $K(d,e)$ is the $(e - 1)$st graded piece of the kernel of $M$, viewed as a map from $S^s$ to $S^r$ that increases degree by $1$, and similarly $K(e,d)^\vee$ is the $e$th graded piece of the cokernel of $M$.
\end{prop}

Because $K(d,e)$ and $K(e,d)$ determine one another, we will usually restrict our attention to $K(d,e)$ for $e \leq d$.

We can further interpret each individual eigenspace of $K(d,e)$ under the torus of diagonal matrices, that is, each individual term of $\kappa(d,e)$. Let $S = \Bbbk[z_1,\ldots, z_n]$ as before, and let $R = S[w_1,\ldots, w_n]$. Consider the $R$-module
\[
  M = \left. R[w_1^{-1},\ldots,w_n^{-1}] \middle/ \sum_i R[w_1^{-1},\ldots,\widehat{w_i}^{-1},\ldots,w_n^{-1}] \right.
\]
Explicitly,
\[
M = \bigoplus_{\substack{a_i,b_j \geq 0}} \Bbbk \cdot \frac{z_1^{b_1} \cdots z_n^{b_n}}{w_1^{1 + a_1} \cdots w_1^{1 + a_n}},
\]
with an $R$-action in which any term with a nonnegative exponent of any $w_i$ is deemed zero. Give $M$ the bigrading
\[
  M_{d,e} = \bigoplus_{\substack{a_i,b_j \geq 0 \\ |\vec a|_1 = d, |\vec b|_1 = e}} \Bbbk \cdot \frac{z_1^{b_1} \cdots z_n^{b_n}}{w_1^{1 + a_1} \cdots w_1^{1 + a_n}}
\]
in which $\deg z_i = (0,1)$, $\deg w_i = (-1,0)$. The bigrading is constructed so that
\[
  M_{d,e} \isom D^d V \tensor S^e V.
\]
Multiplication by $\omega = z_1 w_1 + \cdots + z_n w_n$ induces a map
\[
  M_{d,e-1} \longto^{\omega} M_{d-1,e},
\]
so that we get yet another interpretation of $K(d,e)$:
\begin{equation} \label{eq:monsky_piece}
  K(d,e) = \ker\(M_{d,e-1} \longto^{\omega} M_{d-1,e}\)
\end{equation}
is a graded piece of $K = M/\omega M$.

Moreover, we can decompose $M_{d,e}$ into subspaces
\[
  M_{d,e,\vec u} = \bigoplus_{\substack{a_i,b_j \geq 0 \\ |\vec a|_1 = d, |\vec b|_1 = e, \vec a + \vec b = \vec u}} \Bbbk \cdot \frac{z_1^{b_1} \cdots z_n^{b_n}}{w_1^{1 + a_1} \cdots w_1^{1 + a_n}}.
\]
indexed by nonnegative integer vectors with $|\vec u|_1 = d + e$. Then for $|\vec u|_1 = d + e - 1$,
\begin{equation}
  K(d,e)_{|\vec u|_1} = \ker\(M_{d,e-1} \longto^{\omega} M_{d-1,e}\)
\end{equation}
is the $\vec u$th eigenspace of $K(d,e)$, the space whose dimension is the coefficient of $\!{\vec{\ t\ \!}}^{\!\vec{u}}$ % Default spacing is bad
in $\kappa(d,e)$.

All this is related to the work of Han--Monsky \cite[Theorem 4.5]{HanMonsky}, who find a recurrence for the length of a module of the form
\[
  \Bbbk[z_1,\ldots,z_n]/(z_1^{u_1+1}, z_2^{u_2+1}, \ldots, z_n^{u_n+1}, z_1 + \cdots + z_n).
\]
This is the kernel of multiplication by $\omega$ on a module
\[
  M_{\vec u} = \bigoplus_{\substack{a_i,b_j \geq 0 \\ \vec a + \vec b = \vec u}} \Bbbk \cdot \frac{z_1^{b_1} \cdots z_n^{b_n}}{u_1^{1 + a_1} \cdots u_1^{1 + a_n}}
\]
where $\vec u$ is fixed but $d$ and $e$ can vary, a detail that simplifies the analysis and yields results quite unlike ours.

\subsection{Known cases}

In characteristic zero, the representations $K(d,e)$ are computable using the Borel--Weil--Bott theorem \cite[p.~67 (p.~4 in arXiv version)]{GaoRaicu}: their characters are the Schur functions
\[
\kappa(d,e) = \begin{cases}
  s_{(e - 1, d)} & d < e \\
  0 & \text{otherwise}
\end{cases}
\]
corresponding to partitions $(\lambda_1, \lambda_2)$ with at most two parts.

However, in characteristic $p$, there are additional kernel elements derived from the Frobenius map
\begin{align*}
  F : M &\to M \\
      M_{d,e} &\to M_{p d + (p-1)n,p e} \\
      \frac{c z_1^{b_1} \cdots z_n^{b_n}}{w_1^{1 + a_1} \cdots w_1^{1 + a_n}} & \mapsto \frac{c z_1^{p b_1} \cdots z_n^{p b_n}}{w_1^{p(1 + a_1)} \cdots w_1^{p(1 + a_n)}}.
\end{align*}
Observe that $F$ is $R$-equivariant with the Frobenius map $F : R \to R$ that raises each generator $z_i$, $w_i$ to the $p$th power. In particular, for $\alpha \in M$,
\[
  F(\omega \cdot \alpha) = \omega^p F(\alpha),
\]
whence we get a map
\begin{align*}
  F_K : K &\to K \\
  K(d,e) &\to K\big(p d + (p - 1)(n - 1), p e + p - 1\big) \\
  \alpha &\mapsto \omega^{p-1} F(\alpha).
\end{align*}
In an unpublished preprint, Gao and Raicu observed that $K$ appears to be generated as an $R$-module by lifts of characteristic-zero cocycles (which occur only for $d < e$) and their iterates under $F_K$. Below, we hint at a rather smaller generating set (Conjecture \ref{conj:strip}).

We briefly summarize the cases in which the character $\kappa(d,e)$ is known.
\begin{itemize}
  \item For $e < p$, the character $\kappa(d,e) = s_{e-1,d}$ is the same as in characteristic zero. In particular, it vanishes in the region $d \geq e$ of interest to us.
  \item If $p \leq e < 2p$ and $d \geq e$, the character $\kappa(d,e) = s^{(p)}_{d-1+p, e-p}$ is a truncated Schur function \cite[Theorem 1.6]{GaoRaicu}. (Note that the order of variables differs from that of \cite{GaoRaicu}: our $\kappa(d,e)$ is their $h^0(d,e) = h^1(e,d)$.)
  \item It is known for which $d$ and $e$ we have $K(d,e) = 0$ \cite[Theorem 1.4]{GaoRaicu}; this computes the \emph{Castelnuovo--Mumford regularity} of divided powers of the tautological subbundle $\R$ \cite[Theorem 1.3]{GaoRaicu}.
  \item For $n = 3$ (so $X$ is a complete flag variety), there is a recursive formula for $\kappa(d,e)$; see the next section.
  \item For $p = 2$, an explicit formula in terms of Nim polynomials is known (\cite{KMRR24}; see Theorem \ref{thm:KMRR} below).
\end{itemize}

\section{Even-carry polynomials} \label{sec:even}

For $n = 3$, Gao and Raicu produce a recurrence for $\kappa(d,e)$, originally due to Linyuan Liu who studies it in yet another setting \cite[Th\'eor\`eme 1 and Proposition 9]{Liu2023}.

\begin{thm}[\cite{GaoRaicu}, Theorem 1.7] \label{thm:GR_rec}
  Suppose that $n = 3$, and let $d \geq e > 0$ be integers. Let $1 \leq t < p$ and $k$ be such that $t p^k \leq e < (t + 1) p^k$. Then $\kappa(d,e)$ can be computed as follows:
  \begin{itemize}
    \item If $d \geq (t + 1)p^k - 1$ then $\kappa(d,e) = 0$.
    \item Otherwise, as a character in $A$,
    \begin{equation} \label{eq:recur}
      \begin{aligned}
        \kappa(d,e) &= F^k(s_{(t)}^\vee) \cdot \kappa(d - tp^k, e - tp^k) + F^k(s_{(t-1)}^\vee) \cdot s^{(p^k)}_{d-1+(2-t)p^k, e - t p^k}  \\
        &\quad {} + F^k(s_{(t-2)}^\vee) \cdot \kappa(p^k(t+1) - e - 2, p^k(t+1) - d - 2)^\vee.
      \end{aligned}
    \end{equation}
  \end{itemize}
\end{thm}

\begin{rem} \label{rem:Schur}
  We have written the recurrence using a truncated Schur function as in \cite{GaoRaicu}. However, the truncation lies in the range of Proposition \ref{prop:trunc} and so
  \[
    s^{(p^k)}_{d-1+(2-t)p^k, e - t p^k} = s_{p^k + e - d - 2, (t + 1)p^k - d - 2}
  \]
  is an ordinary Schur function.
\end{rem}

When $p = 2$, a non-recursive formula exists for $\kappa(d,e)$ in terms of the so-called \emph{Nim polynomials}
\begin{equation} \label{eq:Nim_poly}
  N_n(m) = \sum_{\substack{a_1 + \cdots + a_n = 2 m \\ a_1 \oplus \cdots \oplus a_n = 0}} x_1^{a_1} \cdots x_n^{a_n} \in A.
\end{equation}
where $\oplus$ denotes the Nim sum (bitwise XOR, that is, noncarrying addition in base $2$). For $n = 3$, this was shown by Gao and Raicu:

\begin{thm}[\cite{GaoRaicu}, Theorem 1.9]\label{thm:GR_Nim}
  Suppose that $n = 3$, $p = 2$, and $2^k \leq e \leq d < 2^{k+1}$ for some $k \geq 1$. Write the binary expansion $d = (d_k\cdots d_0)_2$, and for $i = 1,\ldots, k$, consider the left and right truncations of the binary expansion of $d$:
  \[
    \ell_i(d) = (d_k \cdots d_{i})_2 = \floor{\frac{d}{2^i}} \textand r_i(d) = (d_{i-1} \cdots d_0)_2 = d - 2^i \ell_{i-1}(d),
  \]
  and define $\ell_i(e), r_i(e)$ analogously. Set
  \[
    I = \{i : 1 \leq i \leq k, d_i = e_i = 1, \ell_i(d) = \ell_i(e), r_i(d) \leq 2^i - 2\}.
  \]
  Then
  \[
    \kappa(d,e) = \sum_{i \in I} F^{i + 1}(N_{\ell_i(e)}) s^{(2^i)}_{r_i(d) - 1 + 2^{i+1}, r_i(e)}.
  \]
\end{thm}

An extension to general $n$ was recently announced by Kyomuhangi, Marangone, Raicu, and Reed, correcting and proving a conjecture of Gao \cite[Conjecture 4.4.1]{GaoThesis}:

\begin{thm}[Kyomuhangi--Marangone--Raicu--Reed; unpublished, \cite{KMRR24}, Theorem 1.3] \label{thm:KMRR} For $p = 2$ and $n \geq 3$,
  \[
  \kappa(d,e) = \sum_{\substack{k \geq 0; m, j \geq 1 \\ 2^{k} (2m + 2j + 1) \leq e}} F^{{k+1}}(N_n(m)) \cdot s^{(2^k)}(d - 2^{k}(2m - 2j - 1), e - 2^k(2m + 2j + 1)),
  \]
  where $N_n(m)$ is the Nim polynomial from \eqref{eq:Nim_poly}.
\end{thm}

 However, the notion of Nim polynomial is not easy to generalize to all $n$ and $p$. In this section, we provide a suitable generalization of Nim polynomials for $n = 3$ and general $p$ to solve the recurrence \eqref{eq:recur} explicitly. 

\begin{defn}
  Fix an integer $p \geq 2$. A finite multiset $\{a_1, a_2, \ldots, a_r\}$ of nonnegative integers is \emph{even-carry} if, for all $i \geq 1$, the $i$th carry
  \[
  c_i = \Bigfloor{\frac{a_1 + \cdots + a_r}{p^i}} - \Bigfloor{\frac{a_1}{p^i}} - \ldots - \Bigfloor{\frac{a_r}{p^i}},
  \]
  which is the carry into the $p^i$ place in working out the sum $a_1 + \cdots + a_r$ in base $p$, is an even integer. We define the $m$th \emph{even-carry polynomial}
  \[
  \C_p(m) \quad = \sum_{\substack{a_j \geq 0 \\ a_1+a_2+a_3+1 = m \\ \{a_1,a_2,a_3,1\}\text{ even-carry}}} x_1^{a_1} x_2^{a_2} x_3^{a_3} \quad \in \quad A.
  \]
\end{defn}
We note the following:
\begin{prop}\label{prop:even_carry}
  The even-carry polynomials $\C_p(m)$ have the following properties:
  \begin{enumerate}[$($a$)$]
    \item \label{ec:small} If $m < p$, then $\C_p(m) = s_{(m-1)}$ is the sum of all monomials of degree $m$.
    \item \label{ec:rec} We have the following recurrence: if $1 \leq t < p$ and $0 \leq m < p^k$, then
    \[
    \C_p(t p^k + m) = F^k s_{(t)} \cdot \C_p(m) + F^k s_{(t - 2)} \C_p(p^k - 1 - m)^\vee.
    \]
    \item \label{ec:digit} In terms of the digit expansion
    \[
    m = \sum_{i = 0}^k d_i p^i, \quad 0 \leq d_i < p, \quad d_k \neq 0,
    \]
    we have the following explicit description:
    \[
    \C_p(m) = \sum_{c_1,\ldots, c_k \in \{0, 2\}} \prod_{i = 0}^k \begin{cases*}
      F^i s_{d_i - c_i} & if $c_{i+1} = 0$ \\
      F^i s_{p - 1 - d_i + c_i}^\vee & if $c_{i+1} = 2$
    \end{cases*}
    \]
    where $c_{k+1} = 0$, $c_0 = 1$.
  \end{enumerate}
\end{prop}
\begin{proof}
  \begin{enumerate}[$($a$)$]
    \item If $m < p$, then the condition $a_1 + a_2 + a_3 + 1 = m$ implies that the $a_i$ are one-digit numbers in base $p$ and there are no carries, so  the even-carry condition is vacuous.
    \item This is proved by dividing into cases based on the last carry $c_k$. If $c_k = 0$, then the terms of $\C_p(t p^k + m)$ can be gotten by combining the possible leading digits (forming $F^k s_{(t)}$) with the rest of each $a_i$ (forming $\C_p(m)$). If $c_k = 2$, then the leading digits now sum to $t - 2$, and the rest of each $a_i$ can form part of a term of $\C_p(2 p^k + m)$, but then we subtract each digit from $p^k - 1$, amounting to the dualization map on $A$. This subtracts all the carries from $2$, so the even-carry property is preserved, and we get $\C_p(p^k - 1 - m)$.
    \item Similarly, this is proved by dividing into cases based on all carries $c_i$. Denote by $a_{ji}$ the $i$th digit of the number $a_j$. The digits $a_{ji}$ must satisfy
    \[
    \sum_{j} a_{ji} = d_i + pc_{i+1} - c_i, \quad 0 \leq a_{ji} < p.
    \]
    The generating function of triples $\{a_j\}_{j=1}^3$ with $\sum_j a_j = h$ fixed and $0 \leq s < p$ is a truncated symmetric function $s^{(p)}_h$; conveniently, the sum $h = d_i + pc_{i+1} - c_i$ lies in regions for which $s_h^{(p)}$ is easy to describe:
    \begin{itemize}
      \item If $c_{i+1} = 0$, then $h \leq p - 1$ and $s^{(p)}_h = s_h$;
      \item If $c_{i+2} = 2$, then $h \geq 2(p - 1)$ and $s^{(p)}_h = s^{(p)\vee}_{3(p-1) - h} = s_{3(p-1) - h}^\vee$.
    \end{itemize}
    Multiplying over $i$ and applying the appropriate Frobenius twists yields the term of the sum corresponding to the choice of carries $c_i$. Note that we set $c_0 = 1$ to account for the added $1$ in the condition $a_1 + a_2 + a_3 + 1 = m$. \qedhere
  \end{enumerate}
\end{proof}
This leads to the following explicit description of the cohomology character:
\begin{thm}\label{thm:Even_Carry}
  Assume that $n = 3$, and let $d \geq e > 0$ be positive integers. % Denote by $r$ the smallest positive integer such that
  % \[
  % \Bigfloor{\frac{d}{p^r}} = \Bigfloor{\frac{e}{p^r}}.
  % \]
  Then
  \begin{equation} \label{eq:h1_even_carry}
    \kappa(d,e) = \sum_{i = 1}^{\floor{\log_p e}} F^i \C_p\Big(\Bigfloor{\frac{e}{p^i}}\Big)^\vee \cdot s_{p^i + e - d - 2, p^i + p^i \floor{\frac{e}{p^i}} - d - 2}.
  \end{equation}
\end{thm}
\begin{proof}
As in Theorem \ref{thm:GR_rec}, we let $1 \leq t < p$ and $k$ be the integers for which $tp^k \leq e < (t + 1)p^k$. Note that the upper bound $i \leq k = \floor{\log_p e}$ is natural, because for $i > \floor{\log_p e}$ we have
  \[
    \C_p\(\Bigfloor{\frac{e}{p^i}}\) = \C_p(0) = 0.
  \]
%Note that the lower bound $i = r$ of the sum may be changed to $0$ without affecting the sum since, if $\floor{d/p^i} > \floor{e/p^i}$, then the second element of the subscript becomes
% \begin{align*}
%  p^i + p^i \floor{\frac{e}{p^i}} - d - 2
%  &\leq p^i + p^i \floor{\frac{e}{p^i}} - \(p^i \floor{\frac{e}{p^i}} - 1\) - 2\\
%  &= -1,
% \end{align*}
%and the Schur function vanishes by our conventions.
  
We induct on $d + e$, taking advantage of the recurrence of Theorem \ref{thm:GR_rec} which expresses $\kappa(d,e)$ in terms of $\kappa(d',e')$ with $d' + e' < d + e$. As a base case, if $d \geq (t + 1)p^k - 1$, then the left side of \eqref{eq:h1_even_carry} vanishes by Theorem \ref{thm:GR_rec}, and the right side vanishes because, for $i \leq k$,
\begin{align*}
    p^i + p^i \floor{\frac{e}{p^i}} - d - 2
    &\leq p^i + p^i \((t+1)p^{k-i} - 1\) - \( (t + 1)p^k - 1\) - 2\\
    &= -1,
\end{align*}
and the Schur function vanishes by our conventions.

So we assume that $d < (t+1)p^k - 1$, implying that $r \leq k = \floor{\log_p e}$. Write
\[
  \kappa(d,e) = \kappa_0(d,e) + \kappa_1(d,e) + \kappa_2(d,e),
\]
where the $\kappa_i$ are the terms of \eqref{eq:recur}:
\begin{equation*} % \label{eq:recur_copy}
  \begin{aligned}
    \kappa_0(d,e) &= F^k(s_{(t)}^\vee) \cdot \kappa(d - tp^k, e - tp^k) \\
    \kappa_1(d,e) &= F^k(s_{(t-1)}^\vee) \cdot s^{(p^k)}_{d-1+(2-t)p^k, e - t p^k} = F^k(s_{(t-1)}^\vee) \cdot s_{p^k + e - d - 2, (t + 1)p^k - d - 2} \quad \text{by Remark \ref{rem:Schur}} \\
    \kappa_2(d,e) &= F^k(s_{(t-2)}^\vee) \cdot \kappa(p^k(t+1) - e - 2, p^k(t+1) - d - 2)^\vee.
  \end{aligned}
\end{equation*}
% \newpage % TODO temporary
We first note that $\kappa_1(d,e)$ coincides with the $i = k$ term of \eqref{eq:h1_even_carry} since $\floor{e/p^k} = t$ lies in the range of Proposition \ref{prop:even_carry}\ref{ec:small}. The other two terms involve $\kappa(d',e')$ where $d', e'$ are less than $p^{k-1}$ by our choice of $k$ and $t$. We have
\begin{align*}
  \kappa_0(d,e) &= F^k(s_{(t)}^\vee) \cdot \kappa(d - tp^k, e - tp^k) \\
  &= F^k(s_{(t)}^\vee) \sum_{i = 1}^{k-1} F^i \C_p\Big(\Bigfloor{\frac{e - tp^k}{p^i}}\Big)^\vee s_{p^i + e - d - 2, p^i + p^i \floor{\frac{e - tp^k}{p^i}} - (d - tp^k) - 2} \\
  &= F^k(s_{(t)}^\vee) \sum_{i = 1}^{k-1} F^i \C_p\Big(\Bigfloor{\frac{e}{p^i}} - t p^{k-i}\Big)^\vee s_{p^i + e - d - 2, \(1 + \floor{\frac{e}{p^i}}\)p^i - d - 2},
\intertext{and}
  \kappa_2(d,e) &= F^k(s_{(t-2)}^\vee) \cdot \kappa(p^k(t+1) - e - 2, p^k(t+1) - d - 2)^\vee \\
  &= F^k(s_{(t-2)}^\vee) \sum_{i = 1}^{k-1} F^i \C_p\Big(\Bigfloor{\frac{p^k(t + 1) - d - 2}{p^i}}\Big) s_{p^i + e - d - 2, p^i + p^i \floor{\frac{(t + 1)p^k - d - 2}{p^i}} - (t + 1)p^k + e}^\vee \\
  &= F^k(s_{(t-2)}^\vee) \sum_{i = 1}^{k-1} F^i \C_p\Big(p^{k-i}(t + 1) - \Bigceil{\frac{d + 2}{p^i}}\Big) s_{p^i + e - d - 2, p^i - p^i \ceil{\frac{d + 2}{p^i}} + e}^\vee \\
  &= F^k(s_{(t-2)}^\vee) \sum_{i = 1}^{k-1} F^i \C_p\Big(p^{k-i}(t + 1) - \Bigceil{\frac{d + 2}{p^i}}\Big) s_{p^i + e - d - 2, p^i\ceil{\frac{d+2}{p^i}} - d - 2},
\end{align*}
where the last step uses the dualization formula
\[
  s_{a,b}^\vee = s_{a,b,0}^\vee = s_{0,-b,-a} = s_{a, a-b, 0} = s_{a, a-b}.
\]

Note that the only difference between the Schur functions appearing in $\kappa_0$ and $\kappa_2$ is the replacement of $1 + \floor{e/p^i}$ by $\ceil{(d+2)/p^i}$. We claim that these integers are equal if either of the Schur functions are nonzero. Since $d \geq e$, we have the bound
\[
  \Bigceil{\frac{d+2}{p^i}} = 1 + \Bigfloor{\frac{d+1}{p^i}} \geq 1 + \Bigfloor{\frac{e}{p^i}}.
\]
Supposing that equality does not hold, we have
\[
  \Bigceil{\frac{d+2}{p^i}} \geq \Bigfloor{\frac{e}{p^i}} + 2.
\]
We find that the Schur function $s_{\lambda_1,\lambda_2}$ appearing in $\kappa_0$ has $\lambda_2 < 0$, while the $s_{\lambda_1,\lambda_2}$ appearing in $\kappa_2$ has $\lambda_1 < \lambda_2$, so both terms are $0$. Consequently, we may replace $\ceil{(d+2)/p^i}$ by $1 + \floor{e/p^i}$ in $\kappa_2$ (including in the argument to $\C_p$), and combining the terms with the same Schur function,
\begin{align*}
  &\kappa_0(d,e) + \kappa_2(d,e)\\
  &{} = \sum_{i=1}^{k-1} \(F^k(s_{(t)}^\vee) F^i \C_p\Big(\Bigfloor{\frac{e}{p^i}} {-t} p^{k-i}\Big)^{\!\vee} \!+  F^k(s_{(t-2)}^\vee) F^i \C_p\Big(p^{k-i}(t{+ 1}){- 1}{- \Bigfloor{\frac{e}{p^i}}}\Big)\) s_{p^i + e - d - 2, \(1 + \floor{\frac{e}{p^i}}\)p^i - d - 2} \\
  &{} = \sum_{i=1}^{k-1} F^i \C_p\Big(\Bigfloor{\frac{e}{p^i}}\Big) s_{p^i + e - d - 2, \(1 + \floor{\frac{e}{p^i}}\)p^i - d - 2},
\end{align*}
where the last step follows from Proposition \ref{prop:even_carry}\ref{ec:rec}. This matches the $1 \leq i \leq k-1$ terms of \eqref{eq:h1_even_carry}, as desired.
\end{proof}
% The proof is immediate, consisting in applying the recurrence of Theorem \ref{thm:GR_rec} and using Proposition \ref{prop:even_carry} to transform the even-carry polynomials. As may be surmised, the notion of even-carry polynomials was devised by examining the data outputted by the recurrence of Theorem \ref{thm:GR_rec}. It is interesting to search for a more illuminating proof.

\subsection{Compatibility with Nim polynomials}
Observe that the $\C_p(m)$ in Theorem \ref{thm:Even_Carry} take the place of the Nim polynomial $\N_3(m)$ in Theorem \ref{thm:GR_Nim}. In fact, we can say more: the even-carry polynomials \emph{are} the Nim polynomials in this overlapping case.

\begin{prop}
  If $p = 2$, then
  \begin{align*}
    \C_p(2m) &= 0 \\
    \C_p(2m + 1) &= F\N_3(m)^\vee.
  \end{align*}
\end{prop}
\begin{proof}
Note that when $p = 2$, there can be no carries in an even-carry multiset. Indeed, a carry of at least $2$ in any column would lead to a carry of at least $1$ in the next column, hence at least $2$ by the even-carry condition, and the carry would have to propagate forever. It is easy to see that $\{a_1,a_2,a_3,1\}$ is an even-carry multiset if and only if all $a_i$ are even and
\[
(a_1 + a_2) \oplus (a_1 + a_3) \oplus (a_2 + a_3) = 0
\]
where $\oplus$ denotes the Nim sum. Thus $\C_p(m)$ vanishes for $m$ even, and
\[
  \C_p(2m+1) = \sum_{\substack{b_j \geq 0 \\ b_1+b_2+b_3 = 2m \\ b_1\oplus b_2\oplus b_3 = 0}} x_1^{-2b_1} x_2^{-2b_2} x_3^{-2b_3} = F\N_3(m)^\vee.
\]
\end{proof}

\section{Conjectures} \label{sec:conj}
\subsection{Primitive cohomology and Prim polynomials} \label{sec:prim}
The foregoing leaves open what should be the appropriate generalization of Nim/even-carry polynomials if $n > 3$ and $p > 2$. Here, we propose a definition.

Recall the interpretation of $K(d,e)$ from \eqref{eq:monsky_piece} above as a graded piece of an $R$-module map (here $R = \Bbbk[z_1,\ldots, z_n, w_1,\ldots, w_n]$),
\[
  K(d,e) = \ker \omega|_{M_{d,e-1}}, \quad
  M_{d,e} = \bigoplus_{|a|_1 = d, |b|_1 = e} \Bbbk \frac{z_1^{b_1} \cdots z_n^{b_n}}{w_1^{1 + a_1} \cdots w_n^{1+a_n}},
\]
and recall the Frobenius map (reminiscent of the Cartier isomorphism for de Rham cohomology)
\begin{align*}
  F_K : K &\to K \\
  K(d,e) &\to K\big(pd + (p - 1)(n - 1), pe + p - 1\big) \\
  \alpha &\mapsto \omega^{p-1} F(\alpha).
\end{align*}
Let $K(d,e)^{\prim}$ be the quotient of $K(d,e)$ by those cycles lying in the $R$-span of the image of $F$.

\begin{conj} \label{conj:strip}
  $K(d,e)^{\prim} = 0$ except when $d = e - 1$.
\end{conj}
We accordingly define $\Prim(m)$ to be the character of $K^\prim(m-1,m)$ for $m \geq 1$. (For example, $\Prim(1) = 1$ because $K(0,1)$ is one-dimensional, fixed by $\SL_n$, and there are no other $K(d,e)$ that can map to it.) We have extensive numerical evidence that this is the correct generalization of both the Nim polynomials and the even-carry polynomials:
\begin{conj}\label{conj:Nim}~
  \begin{enumerate}[$($a$)$]
    \item\label{it:n=3} If $n = 3$, then $\Prim(m)$ is given by an even-carry polynomial:
    \[
      \Prim(m) = \C_p(m).
    \]
    \item\label{it:p=2} If $p = 2$, then $\Prim(m)$ is given by a Nim polynomial:
    \begin{align*}
      \Prim(2m) &= 0 \\
      \Prim(2m + 1) &= N_n(m).
    \end{align*}
  \end{enumerate}
\end{conj}

\begin{rem}
Both the Nim polynomials for $p = 2$ and the even-carry polynomials for $n = 3$ are \emph{multiplicity-free,} that is, are a sum of distinct monomials. This is NOT true of the $\Prim$ polynomials in general. For instance, if $n = 4$ and $p = 3$, then 
\[
\Prim(7) = \sum_{\sym} x_1^6 x_2^6 + \sum_{\sym} x_1^6x_2^3x_3^3 + \sum_{\sym} x_1^4 x_2^4 x_3^4 + \sum_{\sym} x_1^4 x_2^4 x_3^3 x_4 + \boxed{3 x_1^3 x_2^3 x_3^3 x_4^3}
\]
contains a term (boxed) with coefficient $> 1$. (Here $\sum_{\sym} f$ means the symmetric function formed by summing all distinct terms obtained by permuting the variables $x_i$ in $f$.)
\end{rem}

Although the authors of \cite{KMRR24} do not consider $K^\prim$, it is likely that Conjecture \ref{conj:Nim} is within reach of their methodology. In particular, it would be instructive to give a recursive formula for $\Prim(m)$.

\subsection{Tiles and the structure of Prim polynomials} \label{sec:tiles}

In this section, we speculate on the nature of the $\Prim$ polynomials for general $p$ and $n$.

We begin with some data. To save space, we use the following notation.
\begin{defn}
  Given a matrix
  \[
  M = \begin{bmatrix}
    a_{1k} & \cdots & a_{11} & a_{10} \\
    a_{2k} & \cdots & a_{21} & a_{20} \\
    \vdots &        & \vdots & \vdots \\
    a_{nk} & \cdots & a_{n1} & a_{n0}
  \end{bmatrix}
  \]
  of $p$-adic digits $a_{ji}$, with each column $\vec{a}_i = (a_{1i}, a_{2i}, \ldots, a_{ni})$ weakly decreasing, define the \emph{minimal Schur function}
  \[
    \MS(M) = \prod_{i = 0}^k F^i s_{\vec{a}_i}.
  \]
\end{defn}
Observe that $\MS(M)$ is a character whose highest weight is given by reading the rows of $M$ as numbers in base $p$. It is called ``minimal'' because it is the least product of Frobenius images of $s_\lambda$'s (\`a la Steinberg tensor product theorem) achieving this highest weight. Note that not every character is a sum of $\MS$ functions; among other restrictions, the highest weight must be a digit-decreasing partition in base $p$.

\begin{table}
  \begin{alignat*}{7}
    \Prim( 1_5) &=\ & \MS\begin{bsmallmatrix}0\\0\\0\\0\end{bsmallmatrix}\\
    \Prim( 2_5) &=& \MS\begin{bsmallmatrix}1\\1\\0\\0\end{bsmallmatrix}\\
    \Prim( 3_5) &=& \MS\begin{bsmallmatrix}2\\2\\0\\0\end{bsmallmatrix}\\
    \Prim( 4_5) &=& \MS\begin{bsmallmatrix}3\\3\\0\\0\end{bsmallmatrix}\\
    \Prim(10_5) &=&&&&&&& \MS\begin{bsmallmatrix}3\\3\\1\\1\end{bsmallmatrix}\\
    \Prim(11_5) &=& \MS\begin{bsmallmatrix}1&0\\1&0\\0&0\\0&0\end{bsmallmatrix} &&&&&+& \MS\begin{bsmallmatrix}3\\3\\2\\2\end{bsmallmatrix}\\
    \Prim(12_5) &=& \MS\begin{bsmallmatrix}1&1\\1&1\\0&0\\0&0\end{bsmallmatrix} &&&&&+& \MS\begin{bsmallmatrix}3\\3\\3\\3\end{bsmallmatrix}\\
    \Prim(13_5) &=& \MS\begin{bsmallmatrix}1&2\\1&2\\0&0\\0&0\end{bsmallmatrix}\\
    \Prim(14_5) &=& \MS\begin{bsmallmatrix}1&3\\1&3\\0&0\\0&0\end{bsmallmatrix}\\
    \Prim(20_5) &=&&& \MS\begin{bsmallmatrix}1&3\\1&0\\1&0\\0&0\end{bsmallmatrix} &&&+& \MS\begin{bsmallmatrix}1&3\\1&3\\0&1\\0&1\end{bsmallmatrix}\\
    \Prim(21_5) &=& \MS\begin{bsmallmatrix}2&0\\2&0\\0&0\\0&0\end{bsmallmatrix} &+& \MS\begin{bsmallmatrix}1&3\\1&1\\1&1\\0&0\end{bsmallmatrix} &-& \MS\begin{bsmallmatrix}1&2\\1&1\\1&1\\0&1\end{bsmallmatrix} &+& \MS\begin{bsmallmatrix}1&3\\1&3\\0&2\\0&2\end{bsmallmatrix} &+& \MS\begin{bsmallmatrix}1&0\\1&0\\1&0\\1&0\end{bsmallmatrix}\\
    \Prim(22_5) &=& \MS\begin{bsmallmatrix}2&1\\2&1\\0&0\\0&0\end{bsmallmatrix} &+& \MS\begin{bsmallmatrix}1&3\\1&2\\1&2\\0&0\end{bsmallmatrix} &-& \MS\begin{bsmallmatrix}1&2\\1&2\\1&2\\0&1\end{bsmallmatrix} &+& \MS\begin{bsmallmatrix}1&3\\1&3\\0&3\\0&3\end{bsmallmatrix} &+& \MS\begin{bsmallmatrix}1&1\\1&1\\1&0\\1&0\end{bsmallmatrix}\\
    \Prim(23_5) &=& \MS\begin{bsmallmatrix}2&2\\2&2\\0&0\\0&0\end{bsmallmatrix} &+& \MS\begin{bsmallmatrix}1&3\\1&3\\1&3\\0&0\end{bsmallmatrix} &&&&&+& \MS\begin{bsmallmatrix}1&2\\1&2\\1&0\\1&0\end{bsmallmatrix}\\
    \Prim(24_5) &=& \MS\begin{bsmallmatrix}2&3\\2&3\\0&0\\0&0\end{bsmallmatrix} &&&&&&&+& \MS\begin{bsmallmatrix}1&3\\1&3\\1&0\\1&0\end{bsmallmatrix}\\
    \Prim(30_5) &=&&& \MS\begin{bsmallmatrix}2&3\\2&0\\1&0\\0&0\end{bsmallmatrix} &&&+& \MS\begin{bsmallmatrix}2&3\\2&3\\0&1\\0&1\end{bsmallmatrix} &&&+& \MS\begin{bsmallmatrix}1&3\\1&3\\1&1\\1&1\end{bsmallmatrix}\\
    \Prim(31_5) &=& \MS\begin{bsmallmatrix}3&0\\3&0\\0&0\\0&0\end{bsmallmatrix} &+& \MS\begin{bsmallmatrix}2&3\\2&1\\1&1\\0&0\end{bsmallmatrix} &-& \MS\begin{bsmallmatrix}2&2\\2&1\\1&1\\0&1\end{bsmallmatrix} &+& \MS\begin{bsmallmatrix}2&3\\2&3\\0&2\\0&2\end{bsmallmatrix} &+& \MS\begin{bsmallmatrix}2&0\\2&0\\1&0\\1&0\end{bsmallmatrix} &+& \MS\begin{bsmallmatrix}1&3\\1&3\\1&2\\1&2\end{bsmallmatrix}\\
    \Prim(32_5) &=& \MS\begin{bsmallmatrix}3&1\\3&1\\0&0\\0&0\end{bsmallmatrix} &+& \MS\begin{bsmallmatrix}2&3\\2&2\\1&2\\0&0\end{bsmallmatrix} &-& \MS\begin{bsmallmatrix}2&2\\2&2\\1&2\\0&1\end{bsmallmatrix} &+& \MS\begin{bsmallmatrix}2&3\\2&3\\0&3\\0&3\end{bsmallmatrix} &+& \MS\begin{bsmallmatrix}2&1\\2&1\\1&0\\1&0\end{bsmallmatrix} &+& \MS\begin{bsmallmatrix}1&3\\1&3\\1&3\\1&3\end{bsmallmatrix}\\
    \Prim(33_5) &=& \MS\begin{bsmallmatrix}3&2\\3&2\\0&0\\0&0\end{bsmallmatrix} &+& \MS\begin{bsmallmatrix}2&3\\2&3\\1&3\\0&0\end{bsmallmatrix} &&&&&+& \MS\begin{bsmallmatrix}2&2\\2&2\\1&0\\1&0\end{bsmallmatrix}\\
    \Prim(34_5) &=& \MS\begin{bsmallmatrix}3&3\\3&3\\0&0\\0&0\end{bsmallmatrix} &&&&&&&+& \MS\begin{bsmallmatrix}2&3\\2&3\\1&0\\1&0\end{bsmallmatrix}\\
    \Prim(40_5) &=&&& \MS\begin{bsmallmatrix}3&3\\3&0\\1&0\\0&0\end{bsmallmatrix} &&&+& \MS\begin{bsmallmatrix}3&3\\3&3\\0&1\\0&1\end{bsmallmatrix} &&&+& \MS\begin{bsmallmatrix}2&3\\2&3\\1&1\\1&1\end{bsmallmatrix}
  \end{alignat*}
  \caption{Generalized Nim polynomials $\Prim(m)$ for $p = 5$, $n = 4$. The inputs $m$ are written in base $5$ to clarify the patterns.}
  \label{tab:nim_data_p=5}
\end{table}

For $n = 4$, $p = 5$, which we deem a sufficiently general choice of values, we compute the first $20$ $\Prim$ polynomials and find they are all given by minimal Schur decompositions, as shown in Table \ref{tab:nim_data_p=5}. (We scale the terms for $\Prim(m)$ to have degree $2m - 2$; otherwise the highest weights are defined only up to adding $(1,\ldots,1)$.) Moreover, we have aligned the terms to manifest that the whole table is given by the formula
\begin{equation}
  \begin{aligned}
  \Prim(pd_1 + d_0) &= \MS\begin{bsmallmatrix}
    d_1 & d_0 - 1 \\
    d_1 & d_0 - 1 \\
    0 & 0   \\
    0 & 0
  \end{bsmallmatrix}
  + \MS\begin{bsmallmatrix}
    d_1 - 1 & p - 2 \\
    d_1 - 1 & d_0 \\
    1 & d_0 \\
    0 & 0
  \end{bsmallmatrix}
  - \MS\begin{bsmallmatrix}
    d_1 - 1 & p - 3 \\
    d_1 - 1 & d_0 \\
    1 & d_0 \\
    0 & 1
  \end{bsmallmatrix}\\
  &\quad{} + \MS\begin{bsmallmatrix}
    d_1 - 1 & p-2 \\
    d_1 - 1 & p-2 \\
    0 & d_0 + 1 \\
    0 & d_0 + 1
  \end{bsmallmatrix}
  + \MS\begin{bsmallmatrix}
    d_1 - 1 & d_0 - 1 \\
    d_1 - 1 & d_0 - 1 \\
    1 & 0   \\
    1 & 0
  \end{bsmallmatrix}
  + \MS\begin{bsmallmatrix}
    d_1 - 2 & p-2 \\
    d_1 - 2 & p-2 \\
    1 & d_0 + 1 \\
    1 & d_0 + 1
  \end{bsmallmatrix}.
  \end{aligned}
  \label{eq:Nim_formula}
\end{equation}
Here $\MS$ terms where the columns are not weakly decreasing are interpreted as zero, as in the definition of Schur functions above, to account for the gaps in Table \ref{tab:nim_data_p=5}. Similarly, going up to $m < p^3$, we find a formula consisting of $36$ three-column $\MS$ decompositions. Observe that \eqref{eq:Nim_formula} has notable internal structure. For instance, the first and fifth terms, as well as the fourth and sixth, are related by transforming the first column according to the rule
\[
  \begin{bsmallmatrix}
    k\\k\\0\\0
  \end{bsmallmatrix} \mapsto
  \begin{bsmallmatrix}
    k-1\\k-1\\1\\1
  \end{bsmallmatrix}.
\]
Likewise, the second and third terms are related by changing two entries in the last column and flipping the sign. As we collect more and more data, we begin to surmise that $\Prim(m)$ is a linear combination of $\MS$ functions applied to matrices with restricted columns and adjacent column pairs. This motivates the following definition.

% \begin{samepage}
\begin{defn}
  A \emph{tile} $T$ is a tuple $(c_{\OUT}, c_{\IN}, v, \epsilon)$ consisting of the following data:
  \begin{itemize}
    \item An \emph{out-carry} $c_{\OUT}(T) \in \{0, \ldots, n-2\}$;
    \item An \emph{in-carry} $c_{\IN}(T) \in \{0, \ldots, n-2, \END\}$, where ``$\END$'' is a special symbol;
    \item A vector $v(T) = (a_1, \ldots, a_n)$, with $a_i = a_i(T) \in \ZZ$ and $p - 1 \geq a_1 \geq \ldots \geq a_n \geq 0$;
    \item A coefficient $\epsilon(T) \in \ZZ$. (In all known examples, $\epsilon(T) = \pm 1$, but it is risky to conjecture this in general from our limited data.)
  \end{itemize}
%  satisfying the following relations:
%  \begin{enumerate}[$($a$)$]
%    \item If $c_\IN \in \ZZ$, there exists an integer $k$, $0 \leq k \leq p-1$, such that
%    \[
%      \sum_i a_i = pc_\OUT - c_\IN + 2k;
%    \]
%    \item If $c_\IN = \END$, there exists an integer $k$, $0 \leq k \leq p-2$, such that
%    \[
%    \sum_i a_i = pc_\OUT + 2k.
%    \]
%  \end{enumerate}
\end{defn}
% \end{samepage}

\begin{conj} \label{conj:exist tiles}
  For fixed $n$ and $p$, there is a finite list of valid tiles such that the $\Prim$ polynomials $\Prim(d)$ have the form
  \[
    \sum_{\text{tiles } T_0, T_1,\ldots} \prod_i \epsilon(T_i) F^i s_{v(T_i)}
  \]
  where the sum ranges over all sequences of valid tiles $T_0, T_1,\ldots$ with the following properties:
  \begin{enumerate}
    \item All but finitely many of the $T_i$'s are the zero tile
    \[
      \mathbf{0} = (0, 0, (0, \ldots, 0), 1).
    \]
    \item The carries are compatible:
    \[
      c_{\OUT}(T_{i}) = c_{\IN}(T_{i+1}),
    \]
    and $c_{\IN}(T_0) = \END$ (so $\END$ marks tiles that can only go in the $T_0$ position)
    \item The degree is correct:
    \[
      \sum_{i} \sum_j p^i a_j(T_i) = 2d - 2.
    \]
  \end{enumerate}
\end{conj}
%For $n = 2$ there are but two families of valid tiles:
% \[
%  (0, 0, (k, k), 1), \quad 0 \leq k \leq p - 1  \textand (0, \END, (k, k), 1), \quad 0 \leq k \leq p - 2.
% \]
For $n = 3$, the even-carry polynomials satisfy Conjecture \ref{conj:exist tiles} with four families of tiles, listed in Table \ref{tab:tiles_n=3}. Here it is to be noted that the in and out carries of $0$ and $1$ correspond to the carries of $0$ and $2$ in the even-carry multiset, with each $\MS$ term corresponding to a specific carry structure.

For $n = 4$ there appear to be $21$ families of tiles, listed in Table \ref{tab:tiles_n=4}. In these tables, $k$ can be any integer such that $0 \leq k \leq p - 1$ (or $0 \leq k \leq p - 2$, if $c_{\IN} = \END$) and such that the tile entries $a_i$ are weakly decreasing. For example, the lone negative term
\[
  -{\MS\begin{bsmallmatrix}
    d_1 - 1 & p - 2 \\
    d_1 - 1 & d_0 \\
    1 & d_0 \\
    0 & 1
  \end{bsmallmatrix}}
\]
of \eqref{eq:Nim_formula} arises by joining the tiles
\[
  T_0 = (1, 0, (p-3,k,k,1), -1), \quad T_1 = (0, 1, (k,k,1,0), 1), \quad T_i = \mathbf{0} \text{ for } i \geq 2
\]
of Table \ref{tab:tiles_n=4} so that the in-carry $c_\IN = 1$ of the first matches the out-carry $c_\OUT$ of the second. The negative sign on this term comes from the tile $T_1$: $\epsilon(T_1) = -1$.

For $p = 2$, the Nim polynomials also have a tile structure; in other words, Conjecture \ref{conj:Nim}\ref{it:p=2} satisfies Conjecture \ref{conj:exist tiles}. It is unclear how we should assign carry values $c_i$. In the simplest arrangement, the valid tiles are
\[
  \big(0, 0, (\overbrace{1, \ldots, 1}^{2c}, \overbrace{0, \ldots, 0}^{n - 2c}), 1\big), \quad 0 \leq c \leq n/2,
\]
and $\(0, \END, (0,\ldots,0), 1\)$. (In order to satisfy Conjecture \ref{conj:tile_patterns}\ref{it:dual} below, we would have to posit the existence of additional tiles that are inaccessible because $c_\IN$ can never become nonzero.)

Based on the observed structure in tabulated tiles, we conjecture:
\begin{conj} \label{conj:tile_patterns}~
  \begin{enumerate}[$($a$)$]
    \item\label{it:incr} (Incrementing $n$) If $(c_{\OUT}, c_{\IN}, (a_1, \ldots, a_n), \epsilon)$ is a valid tile for a given $n$, then so is 
    \[
      (c_{\OUT}, c_{\IN}, (a_1, \ldots, a_n, 0), \epsilon)
    \]
    for $n$ replaced by $n + 1$.
    \item\label{it:dual} (Duality) If
    \[
      (c_{\OUT}, c_{\IN}, (a_1, \ldots, a_n), \epsilon)
    \]
    is a valid tile with $c_{\IN} \in \ZZ$, so is
    \[
      (n - 2 - c_{\OUT}, n - 2 - c_{\IN}, (p - 1 - a_n, \ldots, p - 1 - a_1), \epsilon);
    \]
    and if
    \[
      (c_{\OUT}, \END, (a_1, \ldots, a_n), \epsilon)
    \]
    is a valid tile, so is
    \[
      (n - 2 - c_{\OUT}, \END, (p - 2 - a_n, \ldots, p - 2 - a_1), \epsilon).
    \]
  \end{enumerate}
\end{conj}

\begin{table}[phtb]
  \begin{tabular}{cccc}
    $c_{\OUT}$ & $c_{\IN}$ & $v$ & $\epsilon$ \\ \hline
    $0$ & $0,\END$ & $(k,k,0)$ & $1$ \\
    $0$ & $1$ & $(k,k,1)$ & $1$ \\
    $1$ & $0,\END$ & $(p-2,k,k)$ & $1$ \\
    $1$ & $1$ & $(p-1,k,k)$ & $1$
  \end{tabular}
  \caption{Tiles for $n = 3$ (the case of even-carry polynomials)}
  \label{tab:tiles_n=3}
\end{table}

\begin{table}[phtb]
  \begin{tabular}{cccc}
    $c_{\OUT}$ & $c_{\IN}$ & $v$ & $\epsilon$ \\ \hline
    $0$ & $0,\END$ & $(k,k,0,0)$ & $1$ \\
    $0$ & $0$ & $(k-1,k-1,1,1)$ & $1$ \\
    $0$ & $1$ & $(k,k,1,0)$ & $1$ \\
    $0$ & $2$ & $(k-1,k-1,0,0)$ & $1$ \\
    $0$ & $2$ & $(k-2,k-2,1,1)$ & $1$ \\
    
    $1$ & $0,\END$ & $(p-2,k,k,0)$ & $1$ \\
    $1$ & $0,\END$ & $(p-3,k,k,1)$ & $-1$ \\
    $1$ & $0$ & $(p-1,k,k,1)$ & $1$ \\
    $1$ & $0$ & $(p-2,k,k,2)$ & $-1$ \\
    $1$ & $1$ & $(p-1,k,k,0)$ & $1$ \\
    $1$ & $1$ & $(p-3,k,k,2)$ & $-1$ \\
    
    $1$ & $2$ & $(p-2,k,k,0)$ & $1$ \\
    $1$ & $2$ & $(p-3,k,k,1)$ & $-1$ \\
    $1$ & $2$ & $(p-1,k,k,1)$ & $1$ \\
    $1$ & $2$ & $(p-2,k,k,2)$ & $-1$ \\
    
    $2$ & $0$ & $(p-1,p-1,k+1,k+1)$ & $1$ \\
    $2$ & $0$ & $(p-2,p-2,k+2,k+2)$ & $1$ \\
    $2$ & $\END$ & $(p-2,p-2,k+1,k+1)$ & $1$ \\
    $2$ & $1$ & $(p-1,p-2,k+1,k+1)$ & $1$ \\
    $2$ & $2$ & $(p-1,p-1,k,k)$ & $1$ \\
    $2$ & $2$ & $(p-2,p-2,k+1,k+1)$ & $1$
  \end{tabular}
\caption{Tiles for $n = 4$}
  \label{tab:tiles_n=4}
\end{table}

\begin{table}[phtb]
  \begin{tabular}{cccc}
    $c_{\OUT}$ & $c_{\IN}$ & $v$ & $\epsilon$ \\ \hline
    
    $2$ & $\END$ & $(p-2,p-2,k+1,k+1,0,0)$ & $1$ \\
    $2$ & $\END$ & $(p-2,p-3,k+1,k+1,1,0)$ & $-1$ \\
    $2$ & $\END$ & $(p-3,p-3,k+1,k+1,2,0)$ & $1$ \\
    $2$ & $\END$ & $(p-2,p-4,k+1,k+1,1,1)$ & $1$ \\
    $2$ & $\END$ & $(p-3,p-4,k+1,k+1,2,1)$ & $-1$ \\
    $2$ & $\END$ & $(p-4,p-4,k+1,k+1,2,2)$ & $1$
  \end{tabular}
  \caption{Tiles for $n = 6$ with $c_{\OUT} = 2$, $c_{\IN} = \END$}
  \label{tab:tiles_n=6}
\end{table}
In Table \ref{tab:tiles_n=6}, we list the tiles found for $n = 6$ having $c_{\OUT} = 2$, $c_{\IN} = \END$, values for which we believe that the list of such tiles is complete. Examination of the tiles found shows that they conform to a uniform pattern
\[
  (2, \END, (p - 2 - \lambda^\top_2, p - 2 - \lambda^\top_1, k+1, k+1, \lambda_1, \lambda_2), (-1)^{\size{\lambda}})
\]
where $\lambda$ and $\lambda^\top$ are conjugate partitions with at most two parts (padded with zeros to have exactly two parts). We conjecture that all the $\END$ tiles are parametrized in this way:
\begin{conj} \label{conj:tiles end}
  The tiles with $c_{\IN} = \END$ are those of the form
  \[
    (c, \END, (p - 2 - \lambda^\top_c, \ldots, p - 2 - \lambda^\top_1, k+c-1, k+c-1, \lambda_1, \ldots, \lambda_{n - 2 - c}), (-1)^{\size{\lambda}})
  \]
  where $\lambda$ is a partition with at most $n - 2 - c$ parts ($c = c_{\OUT}$) whose conjugate $\lambda^\top$ has at most $c$ parts.
\end{conj}
This will not be the first time that conjugate partitions found their way into an algebraic geometry problem of this flavor (see  \cite{kenkel2019lengths,Raicu2018Regularity}).

\subsection{The structure of general \texorpdfstring{$\kappa(d,e)$}{κ(d,e)}}
\label{sec:last}

Due to the complexity of the problem as well as the difficulty in gathering sufficient data, it is harder to conjecture an exact formula for $\kappa(d,e)$. However, all answers computed have the following general form.

\begin{conj}\label{conj:gen}
Fix $n$ and $p$. For each pair $(m,i)$ of integers with
\[
  m \geq 1, \quad 0 \leq i \leq n - 3,
\]
there is a character $P_i(m) \in A$ of degree $2m - 2 + n - i$, such that $P_0(m) = \Prim(m)$ is the $\Prim$ polynomial, and the following formula holds for $d \geq e > 0$:
\begin{equation}\label{eq:gen}
  \kappa(d,e) = \sum_{r,m,i} F^r P_i(m) \cdot s_{e - p^r m, \underbrace{\scriptstyle 0,\ldots,0}_{n-2}, d + p^r(2 - m - n + i) + n - 1}.
\end{equation}
\end{conj}
It is not hard to see that only finitely many choices of $r$, $m$, and $i$ satisfy the needed inequalities for the Schur function shown to be nonzero, namely
\[
  e - p^r m \geq 0 \geq d + p^r(2 - m - n + i) + n - 1.
\]

When $n = 3$, the only terms in \eqref{eq:gen} are indexed by the $\Prim$ polynomials, and on replacing them by the appropriate even-carry polynomials, this reduces to Theorem \ref{thm:Even_Carry}. From Theorem 1.6 of \cite{GaoRaicu}, expanding the truncated Schur function by \eqref{eq:formula for trunc_sym}, we get
\[
  P_i(0) = s_{\underbrace{\scriptstyle1,\ldots,1}_{n - i}}.
\]
Likewise, for $p = 2$, Theorem \ref{thm:KMRR} fits into this framework, expanding the truncated Schur function to get
\begin{equation} \label{eq:Gao}
  P_i(m) = (-1)^i \sum_{\substack{0 \leq j \leq i \\ j \equiv m + 1 \mod 2}} s_{\underbrace{\scriptstyle1,\ldots,1}_{n - i - j}} s_{\underbrace{\scriptstyle1,\ldots,1}_{j}} N_n\(\frac{m - j - 1}{2}\).
\end{equation}
It is tempting to conjecture the same for general $P_i(m)$, replacing the Nim polynomial by $\Prim(m-j)$ and eliminating the congruence condition on $j$; but this does not match the known data for $p = 3$, $n = 4$. It seems difficult to determine the shape of $P_i(m)$ in general.

\section*{Acknowledgments}
I thank Claudiu Raicu, Eric Riedl, Theresa (Tess) Anderson, Steven Sam, Darij Grinberg, and several anonymous referees for their help in various stages of this paper.

\section*{Code}
For readers who would like to experiment further with the questions raised in the paper, Macaulay2 code and documentation can be found in the author's GitHub repository at \url{github.com/emo916math/coho-incidence}.

\bibliography{EvenCarry_bib}
\bibliographystyle{plain}

\textsc{Evan M. O'Dorney, Department of Mathematical Sciences, Carnegie Mellon University, Pittsburgh, PA}
\end{document}